\documentclass{amsart}
\usepackage[utf8x]{inputenc}

\usepackage{amssymb}
\usepackage[active]{srcltx}
\usepackage[all]{xy}
\usepackage{enumerate}
\usepackage{float}
\usepackage{paralist}
\usepackage{cite}
\usepackage{xspace}
\usepackage{multirow}
\usepackage{array}
\usepackage{tikz}
\usepackage{pgf}
\usetikzlibrary{arrows}
\usetikzlibrary{positioning}
\usepackage{pgfkeys}

\usepackage[english]{babel}

\usepackage[colorlinks]{hyperref}

\newcommand{\numberset}[1]{\ensuremath{\mathbb{#1}}} 
\newcommand{\calm}[1]{\mathcal{#1}}
\newcommand{\N}{\numberset{N}} 
\newcommand{\R}{\numberset{R}} 
\newcommand{\Z}[1]{\numberset{Z}_{#1}} 
\newcommand{\Q}{\numberset{Q}} 
\newcommand{\C}{\numberset{C}} 

\newcommand{\Sp}[1]{\ensuremath{\mathbb{S}^{#1}}}

\newcommand{\LF}{L_{F}}
\newcommand{\F}{\calm{F}}

\newcommand{\Ho}[3]{\mathrm{H}_{#1}(#2, #3)}

\newcommand{\pr}[1]{#1^{-1}}

\newcommand{\ad}{a^2 \delta}

\newtheorem{thm}{Theorem}[section]
\newtheorem{prop}[thm]{Proposition}

\newtheorem{cor}[thm]{Corollary}
\theoremstyle{definition}
\newtheorem{rmk}[thm]{Remark}

\newtheorem{defin}[thm]{Definition}

\newcommand{\defi}[1]{\textbf{#1}}

\newcommand{\ie}{\textit{i.e.},\xspace}

\usepackage{color}
\begin{document}
\title{Open book decompositions of $\Sp{5}$ and real singularities}
\author[H. Aguilar-Cabrera]{Hayd\'ee Aguilar-Cabrera}
\address{Department of Mathematics, Columbia University\\
MC 4406\\
2990 Broadway\\
New York, NY\\
10027}
\email{langeh@gmail.com}
\date{March 7, 2012}
\thanks{Research partially supported by CONACyT grant ``Estancias Posdoctorales y Sab\'aticas al Extranjero para la Consolidaci\'on de Grupos de Investigaci\'on''.}
\begin{abstract}
In this article, we study the topology of the family of real analytic germs $F \colon (\C^3,0) \to (\C,0)$ given by $F(x,y,z)=\overline{xy}(x^p+y^q)+z^r$ with $p,q,r \in \N$, $p,q,r \geq 2$ and $(p,q)=1$. Such a germ has isolated singularity at $0$ and gives rise to a Milnor fibration $\frac{F}{|F|} \colon \Sp{5} \setminus L_F \to \Sp{1}$. We describe the link $L_F$ as a Seifert manifold and we show that it is always homeomorphic to the link of a complex singularity. However, we prove that in almost all the cases the open-book decomposition of $\Sp{5}$ given by the Milnor fibration of $F$ cannot come from the Milnor fibration of a complex singularity in $\C^3$.
\end{abstract}

\maketitle

\section{Introduction}
The study of the topology and the geometry of isolated complex singularities is a topic which has generated great interest for many years. A main result is the Milnor Fibration Theorem (see \cite{milnor:singular}): Let $f\colon (\C^n,0) \to (\C,0)$ be a complex analytic germ with isolated singularity; let 
\begin{equation*}
L_f \colon = \pr{f}(0) \cap \Sp{2n-1}_{\varepsilon}
\end{equation*}
be the link of the singularity, where $\Sp{2n-1}_{\varepsilon}$ is a $(2n-1)$-sphere, centred at the origin, of radius $\varepsilon$ sufficiently small. Then the map
\begin{equation*}
\phi_f = \frac{f}{|f|} \colon \Sp{2n-1}_{\varepsilon} \setminus L_f \to \Sp{1}
\end{equation*}
is the projection of a $C^{\infty}$ locally trivial fibration.

Thus $\phi_f$ induces an open-book decomposition of the sphere $\Sp{2n-1}_{\varepsilon}$ with binding $L_f$ and whose pages are the fibres of $\phi_f$.

Furthermore, J.~Milnor shows that some real analytic germs also give rise to fibrations (\cite{milnor:singular}): Let $f \colon (U \subset \R^{n+k},0) \to (\R^k,0)$ be real analytic and a submersion on a punctured neighbourhood $U\setminus\{0\}$ of the origin in $\R^{n+k}$. 
Then there exists a $C^{\infty}$ locally trivial fibration 
\begin{equation*}
\varphi  \colon \Sp{n+k-1}_{\varepsilon} \setminus N(L_f) \to \Sp{k-1} \ .
\end{equation*}

However, the topology of isolated real singularities has not been extensively studied through this result mainly for the following two reasons:

\begin{enumerate}
\item The hypothesis of $f$ having an isolated critical point at the origin is very strong; \ie it is not easy to find examples of real analytic germs satisfying such condition. For example, when $k=2$, the set of critical points of $f$ is, in general, a curve.

\item Even when a function $f$ satisfies this condition, in general it is not true that the projection $\varphi$ can be given by $\frac{f}{||f||}$ as in the complex case.
\end{enumerate}

Nevertheless, now there is an important collection of works giving examples of such real germs and, moreover, studying the geometry and the topology of the fibration $\varphi$. See for example \cite{loo:k2,MR0365592,MR1900787,PichSea:barfg,MR2115674,Cis09,Oka08,MR2647448}.

Some recent results in the area allow us to find examples for which it is possible to give a beautiful and explicit description of the geometry and topology and at the same time to find that the corresponding behaviours are not present in complex singularities. 

For example, consider the family of real analytic germs $F \colon (\C^3,0) \cong (\R^6,0) \to (\C,0) \cong (\R^2,0)$ defined by
\begin{equation}\label{F23}
F(x,y,z)= \overline{xy}(x^2+y^3) + z^r \ \text{with} \ r>2 \ .
\end{equation}
By \cite{Cis09}, $F$ has a Milnor fibration with projection $\phi_F = \frac{F}{|F|}$. In \cite{MR2922705} we prove that the link $\LF$ is a Seifert manifold and that there is not a complex analytic germ $G \colon (\C^3,0) \to (\C,0)$ with isolated singularity at the origin such that the link $L_G$ is homeomorphic to the link $L_F$. On the other hand, given $F \colon (\C^3,0) \cong (\R^6,0) \to (\C,0) \cong (\R^2,0)$ defined by
\begin{equation}\label{F2q}
F(x,y,z)= \overline{xy}(x^2+y^q) + z^2 \ \text{with} \ q>2 \ ,
\end{equation}
in \cite{MR2922705} we have proved that the Milnor fibre $\mathcal{F}$ of $F$ is not the smoothing of a normal Gorenstein complex surface singularity $(X,p)$.

In this paper we study the family of real germs $F \colon (\C^3,0) \to (\C,0)$ given by $F(x,y,z)=\overline{xy}(x^p+y^q)+z^r$ with $p,q,r \in \N$, $p,q,r \geq 2$ and $\gcd(p,q)=1$. What it is done here can be seen as a generalisation of \cite{MR2922705}: We explicitly describe $L_F$ as a Seifert manifold and show that $L_F$ is homeomorphic to the link of a normal complex surface. From there, using plumbing graphs and maximal splice diagrams we compute the canonical class $K$ of a resolution $\widetilde{V}$. We classify the pairs $(p,q)$ such that the coefficients of $K$ are not integers and then $L_F$ cannot be realised as the link of a singularity in $\C^3$, and the pairs $(p,q)$ such that the coefficients of $K$ are integers. In such cases we don't have a complete answer but we show pairs $(p,q)$ where the Milnor fibre of $F$ is not diffeomorphic to any smoothing of a complex singularity in $\C^3$.

The improvement over the techniques used in \cite{MR2922705} is mainly the use of splice diagrams. This tool was developed in \cite{EN85} by Eisenbud and Neumann and greatly used in \cite{MR2140991,MR2140992} by Neumann and Wahl. Through the splice diagrams we have been able to compute the coefficients of the canonical class from the Seifert invariants and the plumbing graph of $L_F$. In  general, the computation of the coefficients of $K$ is done from the resolution graph and the adjunction formula, which gives a system of equations. However, this system can be very big and the increase on its size has a direct relation with the extension in continued fractions of the fraction $\alpha_i/(\alpha_i-\beta_i)$ where ($\alpha_i,\beta_i)$ is a Seifert pair of $L_F$, but with the splice diagrams we do not need to solve the system of equations and easier computations can be done.

\section{The family $F$ of real singularities}

Let $F \colon (\C^3,0) \to (\C,0)$ be given by 
\begin{equation*}
F(x,y,z)=\overline{xy}(x^p+y^q)+z^r
\end{equation*}
with $p,q,r \in \Z{}^{+}$, $p,q,r \geq 2$ and $\gcd(p,q)=1$. By \cite[Corollary~1]{MR2922705}, the function $F$ has isolated singularity.
%

We denote by $L_F=\Sp{5} \cap \pr{F}(0)$ the link of the isolated singularity at the origin of $F$.

\subsection{The function $F$ as a polar weighted homogeneous polynomial}

In this section we will see that the link $L_F$ is a Seifert manifold and that $F$ has Milnor fibration with projection $\displaystyle\frac{F}{|F|}$. For this we will see first that $F$ is a polar weighted homogeneous polynomial.

Polar weighted homogeneous polynomials were introduced by Cisneros-Molina in \cite{Cis09} following ideas presented in \cite{MR1900787} by Ruas, Seade and Verjovsky and studied by Oka in \cite{Oka08} and \cite{oka-2009}.

Let $(p_1, \ldots, p_n)$ and $(u_1, \ldots, u_n)$ in $({\Z{}^+})^n$ be such that $\gcd(p_1, \ldots, p_n)=1$ and $\gcd(u_1, \ldots, u_n)=1$, we consider the action of $\R^+ \times \Sp{1}$ on $\C^n$ defined by:
\begin{equation*}
(t, \lambda) \cdot (z) = (t^{p_1} \lambda^{u_1}z_1, \ldots, t^{p_n}\lambda^{u_n}z_n) \ ,
\end{equation*}
where $t \in \R^+$ and $\lambda \in \Sp{1}$. 

\begin{defin}
Let $f \colon \C^n \to \C$ be a function defined as a polynomial function in the variables $z_i, \bar{z_i}$:
\begin{equation*}
f(z_1, \ldots, z_n) = \sum_{\mu, \nu} c_{\mu, \nu} z^{\mu} \bar{z}^{\nu} \ ,
\end{equation*}
where $\mu= (\mu_1, \ldots, \mu_n)$, $\nu=(\nu_1, \ldots, \nu_n)$, with $\mu_i, \nu_i$ not negative integers and $z^{\mu}=z_1^{\mu_1} \cdots z_n^{\mu_n}$ (same for $\bar{z}$).

The function $f$ is called a \defi{polar weighted homogeneous polynomial} if there exists $(p_1, \ldots, p_n)$ and $(u_1, \ldots, u_n)$ in $({\Z{}^+})^n$ and $a, c \in \Z{}^+$ such that the action defined above satisfies the functional equality:
\begin{equation*}
f((t, \lambda) \cdot (z)) = t^a \lambda^c (f(z)) \ ,
\end{equation*}
where $z=(z_1, \ldots, z_n)$.
\end{defin}

It is known that any polar weighted homogeneous polynomial has an isolated critical value at the origin (see \cite[Proposition~3.2]{Cis09} and \cite[Proposition~2]{Oka08}).

By \cite[Proposition~3]{MR2922705}, the polynomial $f\overline{g} \colon \C^2 \equiv \R^4 \to \C \equiv \R^2$ given by $f(x,y)\overline{g(x,y)}=\overline{xy}(x^p+y^q)$ and the function $F$ are polar weighted homogeneous polynomials. Thus, by \cite[Proposition~3.4]{Cis09}, the function $F$ has Milnor fibration with projection $\phi_F= \frac{F}{|F|}$ since $F$ is a polar weighted homogeneous polynomial.

\subsection{The link $L_F$ as a Seifert manifold}

A \textbf{Seifert manifold} is a closed connected 3-manifold $M$ endowed with a fixed-point free action of $\Sp{1}$; \ie for all $x \in M$, there exist $\lambda \in \Sp{1}$ such that
\begin{equation*}
\lambda \cdot x \neq x \ .
\end{equation*}
Thus, there is an associated fibration $\pi \colon M \rightarrow B$ where $B$  is the orbit surface of the $\Sp{1}$-action.

Let us recall that a Seifert manifold is classified up to isomorphism by its \textbf{Seifert invariants}
\begin{equation*}
\Bigl(g; \, e_0; \, (\alpha_1, \beta_1), \ldots, (\alpha_s, \beta_s)\Bigr)
\end{equation*}
where $g \geq 0$ is the genus of the surface $B$, $(\alpha_i, \beta_i)$ are the exceptional orbits of the action (and $0 <\beta_i<\alpha_i$) and $e_0=e-\sum \beta_i / \alpha_i$ is the rational Euler number of the fibration. The integer $e$ is the Euler class of the $\Sp{1}$-fibration obtained by removing all the exceptional orbits.

By \cite[Proposition~2]{MR2922705}, the link $\LF$ is a Seifert manifold and the next result gives the Seifert invariants of $\LF$:

\begin{thm}[{\cite[Theorem~2]{MR2922705}}]\label{thprinc}
Let $(p,q)$ be coprime integers and $r \in \Z{}^{+}$ with $p,q,r \geq 2$. Let $F \colon \C^3 \to \C$ be the function defined by $F(x,y,z)= \overline{xy}(x^p+y^q)+z^r$. Set $\delta=\gcd(r,pq-p-q)$, $a=r/\delta$ and $b=(pq-p-q)/\delta$. The Seifert invariants of the link $L_F$ are
\begin{equation*}
\Bigl(\frac{\delta-1}{2}; \ - \frac{\delta}{apq}; \ (aq, \beta_1),\,(ap, \beta_2),\,(a, \beta_3)\Bigr) \ ,
\end{equation*}
where
\begin{align*}
b \beta_1 & \equiv -1 \pmod{aq} \ ,\\
b \beta_2 & \equiv -1 \pmod{ap} \ ,\\
b \beta_3 & \equiv 1 \pmod{a} \ .
\end{align*}
\end{thm}

Because of the sign of the rational Euler class, by \cite[Theorem~6]{Neu:calcplumb} we have that there exists a normal complex surface singularity $(V_F,0)$ whose link is homeomorphic to the link $L_F$.

By \cite[Theorem~5.1]{NeuRay:plumb} there exists a plumbing graph $\Gamma_F$ such that $L_F \cong \partial P(\Gamma_F)$, where $P(\Gamma_F)$ is the four-manifold obtained by plumbing $2$-discs bundles according to $\Gamma_F$ (see Figure~\ref{plgrgen}) and
\begin{equation}\label{faccont}
-e= e_0 + \sum_{i=1}^{3} \frac{\beta_i}{\alpha_i}-3 \ , \quad \frac{\alpha_i}{\alpha_i - \beta_i}=[e_{i,1}, \ldots, e_{i,s_i}], \ \text{for} \ i=1,2,3 \ .
\end{equation}
Where, as in \cite{NeuRay:plumb},
\begin{equation*}
[b_1, \ldots, b_k] = b_1 - \cfrac{1}{b_2 - \cfrac{1}{\ddots \cfrac{}{-\cfrac{1}{b_k}}}}
\end{equation*}
denotes the finite continued fraction with $b_1, \ldots, b_k \in \Z{}$.

\begin{figure}[H]
\begin{center}
\begin{tikzpicture}[xscale=.95,yscale=.85]
\filldraw[black, thin] (0,0) circle (1.9pt) node[distance=.5pt and 2pt,below left,scale=1.3] {$-e$}
              -- ++(40:70pt) circle (1.9pt) node [above,scale=1.3] {$-e_{1,1}$}
              -- ++(70pt,0pt) circle (1.9pt) node [above=.5pt,scale=1.3] {$-e_{1,2}$}
            -- ++(20pt,0pt)
               ++(20pt,0pt) circle (.1pt)
               ++(3pt,0pt) circle (.1pt)
               ++(3pt,0pt) circle (.1pt)
               ++(20pt,0pt)
            -- ++(20pt,0pt) circle (1.9pt) node [above=.5pt,scale=1.3] {$-e_{1,s_1}$}
        (0,0) -- (54pt,0pt) circle (1.9pt) node [above,scale=1.3] {$-e_{2,1}$}
              -- ++(70pt,0pt) circle (1.9pt) node [above=.5pt,scale=1.3] {$-e_{2,2}$}
            -- ++(20pt,0pt)
               ++(20pt,0pt) circle (.1pt)
               ++(3pt,0pt) circle (.1pt)
               ++(3pt,0pt) circle (.1pt)
               ++(20pt,0pt)
            -- ++(20pt,0pt) circle (1.9pt) node [above=.5pt,scale=1.3] {$-e_{2,s_2}$}
        (0,0) -- (-40:70pt) circle (1.9pt) node [below=.5pt,scale=1.3] {$-e_{3,1}$}
            -- ++(70pt,0pt) circle (1.9pt) node [below=.5pt,scale=1.3] {$-e_{3,2}$}
            -- ++(20pt,0pt)
               ++(20pt,0pt) circle (.1pt)
               ++(3pt,0pt) circle (.1pt)
               ++(3pt,0pt) circle (.1pt)
               ++(20pt,0pt)
            -- ++(20pt,0pt) circle (1.9pt) node [below=.5pt,scale=1.3] {$-e_{3,s_3}$}
;
\end{tikzpicture}
\end{center}
\caption{Plumbing graph $\Gamma_F$.}
\label{plgrgen}
\end{figure}
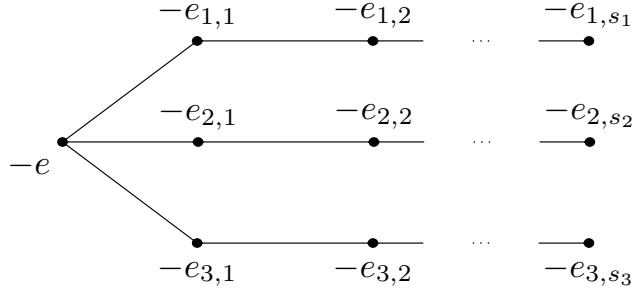

\section{Plumbing graphs and Splice diagrams}

First we recall the concept of a splice and a splice diagram (see \cite[Appendix~1]{MR2140991} and \cite[Section~1]{MR2140992}).

Let $\Sigma_1$ and $\Sigma_2$ be rational homology $3$-spheres and let $K_i \subset \Sigma_i$ (with $i=1,2$) be a knot. Let $N_i$ be a closed tubular neighbourhood of $K_i$ in $\Sigma_i$ for $i = 1, 2$ and let $\Sigma'_i$ be the result of removing its interior, so $\partial \Sigma′_i = T^2$. The \textbf{splice} of $\Sigma_1$ to $\Sigma_2$ along $K_1$ and $K_2$ is the manifold $\Sigma = \Sigma′_1 \cup_{T^2} \Sigma′_2$ where the gluing matches the meridian in $\partial\Sigma_1$ to the longitude in $\partial\Sigma_2$ and vice versa. We denote the splice by
\begin{center}
\begin{tikzpicture}
\draw (-8pt,0pt) node[scale=1.3] {$\Sigma$}
      (5pt,0pt) node[scale=1.3] {$=$}
      (15pt,0pt) node[scale=3] {$($}
      (33pt,0pt) node[left=-5pt,scale=1.3] {$\Sigma_1$}
                 node[above right,scale=1.2] {$K_1$}
   -- (113pt,0pt) node[right=-2pt,scale=1.3] {$\Sigma_2$}
                  node[above left,scale=1.2] {$K_2$}
      (133pt,0pt) node[scale=3] {$)$}
;
\end{tikzpicture}
\end{center}

\begin{defin}
A \textbf{splice diagram} $\Delta$ is a finite tree with no valence $2$ vertices decorated with some integers. The vertices of valence $\geq 3$ are called \textbf{nodes} and the vertices of valence $1$ are called \textbf{leaves}. The tree $\Delta$ is decorated with integer weights as follows: for each node $v$ and edge $e$ incident at $v$, the end of $e$ at $v$ is decorated with an integer $w_{v,e}$. Thus, an edge joining two nodes has weights associated to each end, while an edge from a node to a leaf has just one weight at the node end (see Figure~\ref{wespl}).
\end{defin}

\begin{figure}[H]
\begin{center}
\begin{tikzpicture}[xscale=.78,yscale=.7]
\begin{scope}
\filldraw (-150:50pt) -- (0pt,0pt) circle (1.9pt) node[distance=.5pt and 2pt,below right,scale=1.3] {$w_{v,e}$} -- (80pt,0pt) circle (1.9pt);
\draw (-210:50pt) -- (0pt,0pt);
\end{scope}
\begin{scope}[xshift=200pt]
\filldraw (-150:50pt) -- (0pt,0pt) circle (1.9pt) node[distance=.5pt and 2pt,below right,scale=1.3] {$w_{v,e}$} -- (100pt,0pt) circle (1.9pt) node[distance=.5pt and 2pt,below left,scale=1.3] {$w_{v',e}$};
\draw (-210:50pt) -- (0pt,0pt);
\draw[xshift=100pt] (0pt,0pt) -- (30:50pt)
                   (0pt,0pt) -- (-30:50pt);
\end{scope}
\end{tikzpicture}
\end{center}
\caption{Weights on the splice diagram $\Delta$.}
\label{wespl}
\end{figure}
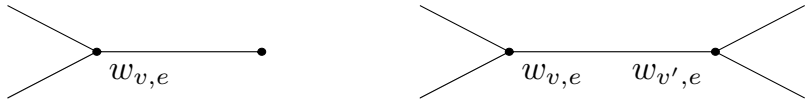

The weight $w_{v,e}$ is computed as follows: Let $v$ be a vertex of $\Delta$ and let $e$ be an edge of $\Delta$ such that $v$ is at one extreme of $e$. Let $\Delta_{v,e}$ be the connected subgraph of $\Delta$ that remains if we eliminate $e$ and the subgraph of $\Delta$ that was connected to $v$; \ie $\Delta_{v,e}$ is the subgraph of $\Delta$ cut off from $v$ (see Figure~\ref{subdelta}). The weight $w_{v,e}$ is the absolute value of the determinant of the intersection matrix for the subgraph $\Delta_{v,e}$.

\begin{figure}[H]
\begin{center}
\begin{tikzpicture}[xscale=.7, yscale=.7]
\begin{scope}
\begin{scope}
\filldraw (-150:50pt) -- (0pt,0pt) circle (1.9pt) node[distance=.5pt and 2pt,below right,scale=1.3] {$w_{v,e}$} node[above=2pt,scale=1.3] {$v$} -- (80pt,0pt) circle (1.9pt);
\draw (-210:50pt) -- (0pt,0pt);
\end{scope}
\draw[yshift=10pt] 
(50pt, -55pt) -- (50pt,-75pt) --(47pt,-70pt) (50pt,-75pt) --(53pt,-70pt);
\begin{scope}[yshift=-110pt]
\filldraw (-150:50pt) --(0pt,0pt) circle (1.9pt) node[above=2pt,scale=1.3] {$v$} (80pt,0pt) circle (1.9pt) node[above=2pt] {$\Delta_{v,e}$};
\draw (-210:50pt) -- (0pt,0pt);
\end{scope}
\end{scope}
\begin{scope}[xshift=250pt]
\begin{scope}
\filldraw (-150:50pt) -- (0pt,0pt) circle (1.9pt) node[distance=.5pt and 2pt,below right,scale=1.3] {$w_{v,e}$} node[above=2pt,scale=1.3] {$v$} -- (100pt,0pt) circle (1.9pt) node[distance=.5pt and 2pt,below left,scale=1.3] {$w_{v',e}$} node[above=2pt,scale=1.3] {$v'$};
\draw (-210:50pt) -- (0pt,0pt);
\draw[xshift=100pt] (0pt,0pt) -- (30:50pt)
                   (0pt,0pt) -- (-30:50pt);
\end{scope}
\draw[yshift=10pt] (50pt, -55pt) -- (50pt,-75pt) --(47pt,-70pt) (50pt,-75pt) --(53pt,-70pt);
\begin{scope}[yshift=-110pt]
\filldraw (-150:50pt) -- (0pt,0pt) circle (1.9pt) node[below=2pt,scale=1.3] {$v$};
\filldraw (100pt,0pt) circle (1.9pt) node[below=2pt,scale=1.3] {$v'$};
\draw (-210:50pt) -- (0pt,0pt);
\draw[xshift=100pt] (0pt,0pt) -- (30:50pt)
                   (0pt,0pt) -- (-30:50pt);
\draw (-10pt,30pt) node {$\Delta_{v',e}$};
\draw (110pt,30pt) node {$\Delta_{v,e}$};
\end{scope}
\end{scope}
\end{tikzpicture}
\end{center}
\caption{Obtaining the subgraph $\Delta_{v,e}$.}
\label{subdelta}
\end{figure}
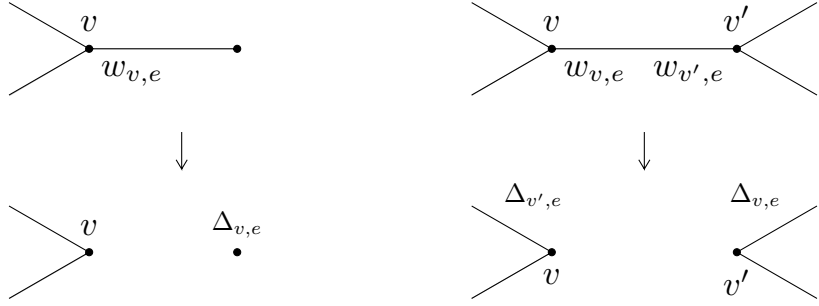

\begin{defin}
The \textbf{edge determinant} $\det(e)$ of an edge $e$ joining two nodes is the product of the two weights on the edge minus the product of the weights adjacent to the edge.
\end{defin}
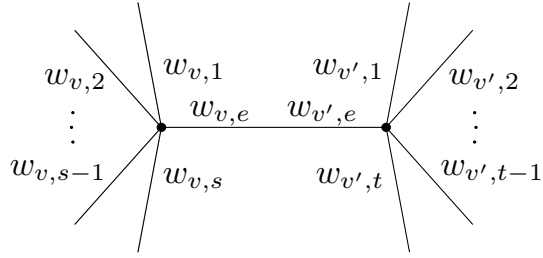
\begin{figure}[H]
\begin{center}
\begin{tikzpicture}[xscale=.85, yscale=.8]
\filldraw (100:60pt) -- (0pt,0pt) circle (1.9pt) (130:60pt) -- (0pt,0pt) (230:60pt) -- (0pt,0pt) (260:60pt) -- (0pt,0pt)  (0pt,0pt)-- ++(100pt,0pt) circle (1.9pt);
\draw (0pt,0pt) node[above right=-4pt and 6pt, scale=1.3] {$w_{v,e}$}
                node[above right=12pt and -4pt, scale=1.3] {$w_{v,1}$}
                node[above left=8pt and 16pt, scale=1.3] {$w_{v,2}$}
                node[below left=8pt and 16pt, scale=1.3] {$w_{v,s-1}$}
                node[below right=11pt and -4pt, scale=1.3] {$w_{v,s}$};
\filldraw (-40pt,0pt) circle (.5pt)
          (170:40pt) circle (.5pt)
          (190:40pt) circle (.5pt); 
\draw[xshift=100pt] (80:60pt) --( 0pt,0pt) (50:60pt) --( 0pt,0pt) (-50:60pt) --( 0pt,0pt) (-80:60pt) --( 0pt,0pt);
\draw (100pt,0pt) node[above left=-4pt and 6pt, scale=1.3] {$w_{v',e}$}
                node[above left=12pt and -4pt, scale=1.3] {$w_{v',1}$}
                node[above right=8pt and 19pt, scale=1.3] {$w_{v',2}$}
                node[below right=8pt and 16pt, scale=1.3] {$w_{v',t-1}$}
                node[below left=11pt and -4pt, scale=1.3] {$w_{v',t}$};
\filldraw[xshift=100pt] (40pt,0pt) circle (.5pt)
          (10:40pt) circle (.5pt)
          (-10:40pt) circle (.5pt);
\end{tikzpicture}
\end{center}
\caption{Edge $e$ joining nodes $v$ and $v'$.}
\label{detedge}
\end{figure}
In Figure~\ref{detedge}, the determinant edge of the edge $e$ that joins nodes $v$ and $v'$ is given by
\begin{equation*}
\det(e)= (w_{v,e})(w_{v',e})-(w_{v,1})(w_{v,s})\cdots(w_{v',1})(w_{v',t}) \ .
\end{equation*}

By \cite[Proposition~7.3]{EN85}, the splice diagrams that arise in the study of links of complex singularities always satisfy the following conditions:
\begin{itemize}
\item All weights are positive.
\item All edge determinants are positive.
\end{itemize}

In this work we will need a variant of splice diagrams where valency $2$ vertices are allowed, and weights are also associated to the leaf end of an edge ending in a leaf.

\begin{defin}
A \textbf{maximal splice diagram} is a splice diagram where vertices of valency 2 are allowed and there are included edge weights at all vertices; \ie the maximal splice diagram can contain bamboos (chains of valence 2 vertices) and it has weights on the leaves.
\end{defin}

The way to compute the weight $w_{v,e}$ is to take the numerator of the generalized continued fraction on the graph $\Delta_{v,e}$. Let us give an example using that $\Gamma_F$ can be seen as a maximal splice diagram $\Delta_F$ with weights computed from the intersection numbers on $\Gamma_F$.
\begin{figure}[H]
\begin{center}
\begin{tikzpicture}[xscale=.85, yscale=.7]
\filldraw[black, thin]
           (-42pt,0pt) node[scale=1.3] {$\Delta_F$}
           (-20pt,0pt) node[scale=1.3] {$=$}
           (0,0) circle (1.9pt) node[above right=8pt and -11pt,scale=1.3] {$w_1$}
                                node[above right=-2pt and 10pt,scale=1.3] {$w_2$}
                                node[below=10pt,scale=1.3] {$w_3$}
      -- ++(50:90pt) circle (1.9pt) node[below left=-2pt and 15pt,scale=1.3] {$w_{1,1}^I$}
                                    node[above right=-3pt and 2pt,scale=1.3] {$w_{1,1}^O$}
      -- ++(90pt,0pt) circle (1.9pt) node[above left=-3pt and 0pt,scale=1.3] {$w_{1,2}^I$}
                                     node[above right=-3pt and 2pt,scale=1.3] {$w_{1,2}^O$}
      -- ++(45pt,0pt)
         ++(15pt,0pt) circle (.1pt)
         ++(3pt,0pt) circle (.1pt)
         ++(3pt,0pt) circle (.1pt)
         ++(15pt,0pt)
      -- ++(45pt,0pt) circle (1.9pt) node [above left=-3pt and 0pt,scale=1.3] {$w_{1,s_1}$}
(0,0) -- ++(0:90pt) circle (1.9pt) node[above left=-2pt and 1pt,scale=1.3] {$w_{2,1}^I$}
                                     node[above right=-2pt and 2pt,scale=1.3] {$w_{2,1}^O$}
      -- ++(90pt,0pt) circle (1.9pt) node [above left=-2pt and 0pt,scale=1.3] {$w_{2,2}^I$}
                                     node [above right=-2pt and 2pt,scale=1.3] {$w_{2,2}^O$}
      -- ++(45pt,0pt)
         ++(15pt,0pt) circle (.1pt)
         ++(3pt,0pt) circle (.1pt)
         ++(3pt,0pt) circle (.1pt)
         ++(15pt,0pt)
      -- ++(45pt,0pt) circle (1.9pt) node [above left=-2pt and 0pt,scale=1.3] {$w_{2,s_2}$}
  (0,0) -- (-50:90pt) circle (1.9pt) node[above left=0pt and 9pt,scale=1.3] {$w_{3,1}^I$}
                                     node[below right=0pt and 2pt,scale=1.3] {$w_{3,1}^O$}  
      -- ++(90pt,0pt) circle (1.9pt) node[below left,scale=1.3] {$w_{3,2}^I$}
                                     node[below right=0pt and 2pt,scale=1.3] {$w_{3,2}^O$}
      -- ++(45pt,0pt)
         ++(15pt,0pt) circle (.1pt)
         ++(3pt,0pt) circle (.1pt)
         ++(3pt,0pt) circle (.1pt)
         ++(15pt,0pt)
      -- ++(45pt,0pt) circle (1.9pt) node [below left,scale=1.3] {$w_{3,s_3}$}
;
\end{tikzpicture}
\end{center}
\caption{Maximal splice diagram $\Delta_F$.}
\label{spldiaggen}
\end{figure}
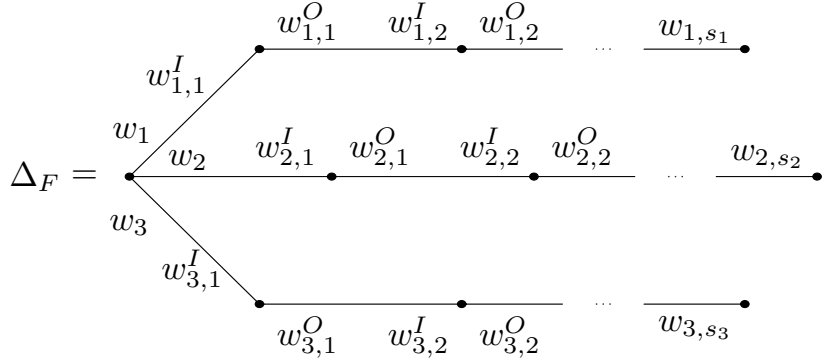
In Figure~\ref{spldiaggen} the weights at the node are labelled $w_1$,$w_2$ and $w_3$. Now, let $v_{i,j}$ the vertex with weight $-e_{i,j}$ in Figure~\ref{plgrgen}. The weights in $\Delta_F$ at the vertex $v_{i,j}$ of valence two are labelled $w_{i,j}^I$ and $w_{i,j}^O$, where $w_{i,j}^I$ is the weight ``inside''; \ie it is in the path connecting the node and the vertex $v_{i,j}$, and $w_{i,j}^O$ is the weight ``outside''; \ie it is in the path that connects $v_{i,j}$ with the end of the bamboo.

The weight $w_i$ is given by the numerator of the continued fraction $[e_{i,1}, \ldots, e_{i,s_i}]$, for $i=1,2,3$. Thus, by equation~\eqref{faccont}, the weight $w_i$ is given by
\begin{equation}\label{wealp}
w_i = \alpha_i \ .
\end{equation}
Analogously, the weight $w_{i,j}^O$ is given by the numerator of the continued fraction $[e_{i,j+1}, \ldots, e_{i,s_i}]$ for $i=1,2,3$ and $1 \leq j \leq s_i-1$.

Now, let us compute the weights $w_{i,j}^I$ for $i=1,2,3$ and $1 \leq j \leq s_i-1$. Let $\hat \omega_1$, $\hat \omega_2$ and $\hat \omega_3$ be given by
\begin{align*}
\hat \omega_1 & = e - [e_{2,1}, \ldots, e_{2,s_2}] - [e_{3,1}, \ldots, e_{3,s_3}] \\
\hat \omega_2 & = e - [e_{1,1}, \ldots, e_{1,s_1}] - [e_{3,1}, \ldots, e_{3,s_3}] \\
\hat \omega_3 & = e - [e_{1,1}, \ldots, e_{1,s_1}] - [e_{2,1}, \ldots, e_{2,s_2}]  \ . \\
\end{align*}
Then we have that the weight $w_{i,j}^I$ is given by the numerator of the continued fraction $w_{i,j}^I = [e_{i,j-1}, \ldots, e_1, \hat \omega_i]$ for $i=1,2,3$ and $1 \leq j \leq s_i-1$.


The computation of the weights $w_{1.s_1}$, $w_{2,s_2}$ and $w_{3,s_3}$ will be done in the next section, since it is necessary to talk about the resolution graph of the singularity.

\section{New Open books}

The plumbig graph $\Gamma_F$ determines a $4$-dimensional compact manifold $\widetilde V_F$ such that the boundary $\partial \widetilde V_F$ is homeomorphic to $L_F$.

The manifold $\widetilde V_F$ contains in its interior a exceptional divisor $E$, given by the plumbing description and such that $\widetilde V_F$ has $E$ as a strong deformation retract. In this setting, the intersection matrix corresponds to the intersection product in $\Ho{2}{\widetilde V_F}{\Q} \cong \Q^n$, where $n$ is the number of vertices in the graph. Thus, the generators are the irreducible components $E_i$ of the divisor $E$; each of these is a compact (non-singular) Riemann surface embedded in $\widetilde V_F$ and the product $E_i \cdot E_j$ corresponds to the entry $a_{i,j}$ of $A_F$. 

The fact that the matrix $A_F$ is negative definite implies, by Grauert's contractibility criterion (see for instance \cite[Section~III, Theorem~2.1]{ccs:BPVdV}), that the divisor $E$ can be blown down to a point, and we get a complex surface $V_F$ with normal singularity at $0$, given by \cite[Corollary~3]{MR2922705}, as the image of the divisor $E$.

Recall that the \textbf{canonical class} $K$ of a resolution $\widetilde V$ is the unique class in $\Ho{2}{\widetilde{V}}{\Q} \cong \Q^n$ that satisfies the \textit{adjunction formula}
\begin{equation*}
2g_i-2=E_i^2 + K \cdot E_i
\end{equation*}
for each irreducible component $E_i$ of the exceptional divisor (see \cite[Lemma~1.4]{DAH78}). Since the canonical class $K$ is by definition a homology class in $\Ho{2}{\widetilde{V}}{\Q}$, it is a rational linear combination of the generators: 
\begin{equation*}
 K = \sum_{i=1}^n k_i \, E_i \ , \text{with} \ k_i \in \Q \ .
\end{equation*}

\begin{defin}(see for instance \cite[Definition~1.2]{DAH78})
A normal surface singularity germ $(V,0)$ is \textbf{Gorenstein} if there is a nowhere-zero holomorphic two-form on the regular points of $V$. In other words, its canonical bundle $\mathcal{K}:= \bigwedge^2 \big(T^*(V\setminus\{0\}) \big)$ is holomorphically trivial in a punctured neighbourhood of $0$.
\end{defin}

Thus every isolated hypersurface singularity is Gorenstein, since if $V$ can be defined by a holomorphic map-germ $f$ in $\C^3$, then the gradient $\nabla f$ never vanishes away from $0$ and we can contract the holomorphic 3-form $dz_1 \wedge dz_2 \wedge dz_3$ with respect to $\nabla f$ to get a never vanishing holomorphic 2-form on a neighbourhood of $0$ in $ V$. More generally, A. Durfee in \cite{DAH78} introduced the following concept.

\begin{defin} The singularity germ $(V,0)$ is \textbf{numerically Gorenstein} if the canonical class $K$ of some good resolution $\widetilde V$ is integral.
\end{defin}

Let us recall that a resolution $\widetilde V$ is  good if each irreducible component $E_i$ is non-singular, and if $i \neq j$, $E_i$ intersects $E_j$ in at most one point where they met transversely and no three irreducible components intersect.

\begin{rmk}\label{integral}
It is easy to show (see \cite[Lemma~1.3]{DAH78} or \cite[Lemma~2.3]{MR1359516}) that the condition of $K$ being an integral class  is satisfied if and only if the canonical bundle $\mathcal{K}$ is topologically trivial. Therefore numerically Gorenstein is a condition independent of the choice of resolution (whose dual graph is $\Gamma$)  on the germ $(V,0)$ and every Gorenstein germ is numerically Gorenstein. 
\end{rmk}

Let us return to the case of our interest. Given the intersection matrix $A_F$, let $\widetilde{V}_F$ be a resolution of the corresponding complex surface singularity $(V_F,p)$ such that $L_F$ is homeomorphic to the boundary of $\widetilde{V}_F$. Let $E$ be the exceptional divisor contained in $\widetilde{V}_F$.

Given the irreducible components $E_i$ (with $1 \leq i \leq n$) of the exceptional divisor, let $\varrho_i$ be the number of intersection points of $E_i$ with $E_j$ for all $j \neq i$ and let $\hat{E}_i$ be $E_i$ with these intersection point removed. Denote
\begin{equation*}
\chi_i=\chi(\hat{E_i}) = \chi(E_i)-\varrho_i \ .
\end{equation*}

Let $K$ be the canonical class of the resolution $\widetilde{V}_F$, let
\begin{equation*}
D = -K - E \quad \text{where} \quad E=\sum_{i=1}^n E_i
\end{equation*}
and let $d_i \in \Q$ be such that
\begin{equation*}
D = \sum_{i=1}^n d_i E_i \ .
\end{equation*}

Thus the adjunction formula becomes:
\begin{align*}
D \cdot E_j &= (-K-E) \cdot E_j = -K \cdot E_j - E \cdot E_j = \chi(E_j) + E_j^2 - E \cdot E_j \\
            &= \chi(E_j) + E_j \cdot E_j - \sum_{i=1}^n E_i \cdot E_j  = \chi(E_j) + E_j \cdot E_j - E_j
               \cdot E_j - \varrho_j\\
            &= \chi(E_j) - \varrho_j \\
            &= \chi_j \ .
\end{align*}
Also
\begin{equation*}
D \cdot E_j = \left(\sum_{i=1}^n E_i\right) \cdot E_j = \sum_{i=1}^n d_i(E_i \cdot E_j)
\end{equation*}
and from the adjunction formula we obtain
\begin{align*}
d_1 E_1 \cdot E_1 + \cdots + d_n E_n \cdot E_1 &= \chi_1 \\
d_1 E_1 \cdot E_2 + \cdots + d_n E_n \cdot E_2 &= \chi_2 \\
\vdots & \\
d_1 E_1 \cdot E_n + \cdots + d_n E_n \cdot E_n &= \chi_n  \ , \\
\end{align*}
let us recall that the intersection numbers $E_i \cdot E_j$ are the entries of the matrix $A_F$, then
\begin{equation*}
(d_1, \ldots, d_n) A_F = (\chi_1, \ldots, \chi_n) \ ,
\end{equation*}
\ie
\begin{equation*}
(d_1, \ldots, d_n) = (\chi_1, \ldots, \chi_n) A_F^{-1}
\end{equation*}
from where:
\begin{equation}\label{coeffd}
d_i= - \sum_{j=1}^n l_{ij} \chi_j \quad \text{where} \quad (l_{ij})=-A_F^{-1} \ .
\end{equation}

Now, as
\begin{align*}
K &= -D -E = -\sum_{i=1}^n d_i E_i - \sum_{i=1}^n E_i = \sum_{i=1}^n (-d_i-1) E_i \\
  &= \sum_{i=1}^n \left( \sum_{j=1}^n l_{ij} \chi_j -1 \right) E_i \ ,
\end{align*}
then
\begin{equation}\label{kend}
k_i = -d_i-1 =\sum_{j=1}^n l_{ij} \chi_j -1 \quad \text{for all} \quad i=1, \ldots, n \ .
\end{equation}
Thus, the canonical class $K$ has integer coefficients if and only if the coefficients of $D$ are integers.

Now let us finish with the computations of the weights of the splice diagram $\Delta_F$: By \cite[Appendix~1]{MR2140991}, we have that
\begin{equation*}
\det(\Gamma_F)=\det(-A_F)= |\Ho{1}{L_F}{\mathbb{Z}}| = a^2 \delta \ .
\end{equation*}

Thus, applying \cite[Theorem~12.5]{MR2140991}, we can compute the weight $w_{i,s_i}$ at the $i$-th leaf:
\begin{equation}\label{leaves}
\begin{aligned}
w_{1.s_1} &= \frac{b}{aq} (\ad) + \frac{a^2 q}{aq} =a(p-1)\\
w_{2,s_2} &= \frac{b}{ap} (\ad) + \frac{a^2 p}{ap} =a(q-1)\\
w_{3,s_3} &= \frac{b}{a}  (\ad) + \frac{a^2 pq}{a} = br+apq \ . \\ 
\end{aligned}
\end{equation}

\begin{rmk}\label{detpath}
\begin{enumerate}
\item In a plumbing graph, $\chi_j=2 - \varrho_j = 0$ except at nodes and leafs.
\item By \cite[Theorem~9.1]{MR2140992}, if $i \neq j$ then $l_{ij}$ is the product of the weights adjacent to, but not on, the shortest path from $i$ to $j$ in $\Delta_F$, where $(l_ij)= -A_F^{-1}$.
\end{enumerate}
\end{rmk}

The following result is the main theorem of the present article. It classifies if the canonical class $K$ is integral or not, in terms of the exponents on the function $F$.

\begin{thm}\label{coeffint}
Let $F \colon (\C^3,0) \cong (\R^6,0) \to (\C,0) \cong (\R^2,0)$ be the real polynomial function defined by
\begin{equation*}
F(x,y,z)= \overline{xy}(x^p+y^q) + z^r \ ,
\end{equation*}
where $(p,q)$ are coprime integers and $r \in \Z{}^{+}$ with $p,q,r \geq 2$. Let $\delta=(pq-p-q,r)$, $a=r/\delta$ and $b=(pq-p-q)/\delta$. 

Let $\Gamma_F$ be the plumbing graph such that $L_F=\partial P(\Gamma_F)$. Let $\widetilde{V}$ be a resolution such that its graph is the dual of $\Gamma_F$ and let $K$ be the canonical class of $\widetilde{V}$. Then,
\begin{enumerate}[i)]
\item If $a \neq 2$, then the coefficients of $K$ are not integers; \ie the complex analytic germ $(V,p)$, such that $L_V \cong L_F$, is not numerically Gorenstein.
\item If $a = 2$, then the coefficients of $K$ are integers; \ie the complex analytic germ $(V,p)$, such that $L_V \cong L_F$, is numerically Gorenstein.
\end{enumerate}
\end{thm}

\begin{proof}
Let us separate the proof in two cases:
\begin{enumerate}[i)]
 \item When $a=1$, $b \in \Z{}^+$,
 \item when $a > 1$, $b \in \Z{}^+$.
\end{enumerate}

\subsection*{Case (i)} When $a=1$, $b \in \Z{}^+$.

In this case the Seifert invariants of $L_F$ are:
\begin{equation*}
\Bigl(\frac{\delta-1}{2}; \ - \frac{\delta}{pq}; \ (q, \beta_1),\,(p, \beta_2)\Bigr) \ ,
\end{equation*}
where
\begin{align*}
b \beta_1 & \equiv -1 \pmod{q} \ ,\\
b \beta_2 & \equiv -1 \pmod{p} \ . \\
\end{align*}
The corresponding plumbing graph $\Gamma_F$ is given in Figure~\ref{pmbgra1}.
\begin{figure}[H]
\begin{center}
\begin{tikzpicture}[xscale=.6, yscale=.5]
\filldraw[black, thin]
           (-90pt,0pt) node[scale=1.4] {$\Gamma_F$}
           (-60pt,0pt) node[scale=1.4] {$=$}
           (0,0) circle (1.9pt) node[left=4pt,scale=1.4]{$-e$}
      -- ++(60:90pt) circle (1.9pt) node[above,scale=1.4]{$-e_{1,1}$}
      -- ++(90pt,0pt) circle (1.9pt) node[above,scale=1.4]{$-e_{1,2}$}
      -- ++(45pt,0pt)
         ++(15pt,0pt) circle (.1pt)
         ++(3pt,0pt) circle (.1pt)
         ++(3pt,0pt) circle (.1pt)
         ++(15pt,0pt)
      -- ++(45pt,0pt) circle (1.9pt) node[above,scale=1.4]{$-e_{1,s_1}$}
(0,0) -- ++(0:90pt) circle (1.9pt) node[below,scale=1.4]{$-e_{2,1}$}
      -- ++(90pt,0pt) circle (1.9pt) node[below,scale=1.4]{$-e_{2,2}$}
      -- ++(45pt,0pt)
         ++(15pt,0pt) circle (.1pt)
         ++(3pt,0pt) circle (.1pt)
         ++(3pt,0pt) circle (.1pt)
         ++(15pt,0pt)
      -- ++(45pt,0pt) circle (1.9pt) node[below,scale=1.4]{$-e_{2,s_2}$}
;
\end{tikzpicture}
\end{center}
\caption{Plumbing graph $\Gamma_F$ when $a=1$.}
\label{pmbgra1}
\end{figure}
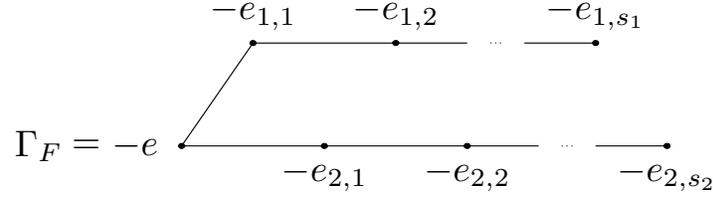

Thus, $w_1=q$ and $w_2=p$ since
\begin{align*}
\frac{q}{q-\beta_1} &= [e_{1,1}, \ldots, e_{1,s_1}] \ , \\
\frac{p}{p-\beta_2} &= [e_{2,1}, \ldots, e_{2,s_2}] \
\end{align*}
and $w_{1,s_1}=p-1$ and $w_{2,s_2}=q-1$ by equation~\eqref{leaves}.

The associated maximal splice diagram $\Delta_F$ is given in Figure~\ref{spldauno}.
\begin{figure}[H]
\begin{center}
\begin{tikzpicture}[xscale=.8, yscale=.7]
\filldraw[black, thin]
           (-65pt,0pt) node[scale=1.4] {$\Delta_F$}
           (-35pt,0pt) node[scale=1.4] {$=$}
           (0,0) circle (1.9pt) node[above right=8pt and -8pt,scale=1.4] {$q$}
                                node[below right=0pt and 8pt,scale=1.4] {$p$}
      -- ++(60:90pt) circle (1.9pt) node[below left=1pt and 8pt,scale=1.4] {$w_{1,1}^I$}
                                    node[above right=-3pt and 2pt,scale=1.4] {$w_{1,1}^O$}
      -- ++(65pt,0pt)
         ++(15pt,0pt) circle (.1pt)
         ++(3pt,0pt) circle (.1pt)
         ++(3pt,0pt) circle (.1pt)
         ++(15pt,0pt)
      -- ++(75pt,0pt) circle (1.9pt) node [above left=-3pt and 0pt,scale=1.3] {$(p-1)$}
(0,0) -- ++(0:90pt) circle (1.9pt) node[above left=1pt and 1pt,scale=1.4] {$w_{2,1,l}$}
                                     node[above right=-2pt and 2pt,scale=1.4] {$w_{2,1,r}$}
      -- ++(65pt,0pt)
         ++(15pt,0pt) circle (.1pt)
         ++(3pt,0pt) circle (.1pt)
         ++(3pt,0pt) circle (.1pt)
         ++(15pt,0pt)
      -- ++(75pt,0pt) circle (1.9pt) node [above left=-2pt and 0pt,scale=1.3] {$(q-1)$}
;
\end{tikzpicture}
\end{center}
\caption{Maximal splice diagram $\Delta_F$.}
\label{spldauno}
\end{figure}
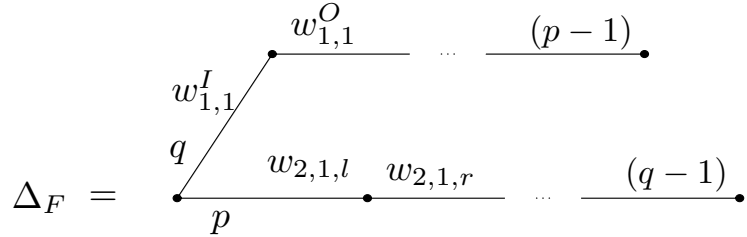

Let $d_1$ and $d_2$ be the entries of matrix $D$ corresponding to the vertices of valence $1$ (the leaves of the diagram). Using equation~\eqref{coeffd} and Remark~\ref{detpath} we compute $d_1$ and $d_2$ and we obtain
\begin{align*}
d_{1} &= - \left( \frac{p-1}{\delta} + \frac{1}{\delta}\right) = - \frac{p}{\delta} \\
d_{2} &= - \left( \frac{q-1}{\delta} + \frac{1}{\delta}\right) = - \frac{q}{\delta} \\
\end{align*}
As $(p,q)=1$, $\delta$ cannot divide $p$ and $q$, then $d_{1}$ and $d_{2}$ cannot be integers at the same time.

\subsection*{Case (ii)} When $a > 1$, $b \in \Z{}^+$.
In this case the Seifert invariants of the link $L_F$ are:
\begin{equation*}
\Bigl(\frac{\delta-1}{2}; \ - \frac{\delta}{apq}; \ (aq, \beta_1),\,(ap, \beta_2),\,(a,\beta_3)\Bigr) \ .
\end{equation*}

By Figure~\ref{spldiaggen} and formulae~\eqref{wealp} and~\eqref{leaves}, we have that
\begin{align*}
 w_1&=aq & w_2&=ap & w_3&=a \\
w_{1,s_1}&=a(p-1) & w_{2,s_2}&=a(q-1) & w_{3,s_3}&=br+apq \ .
\end{align*}
Thus the corresponding splice diagram $\Delta_F $ is given by Figure~\ref{spldpart}.
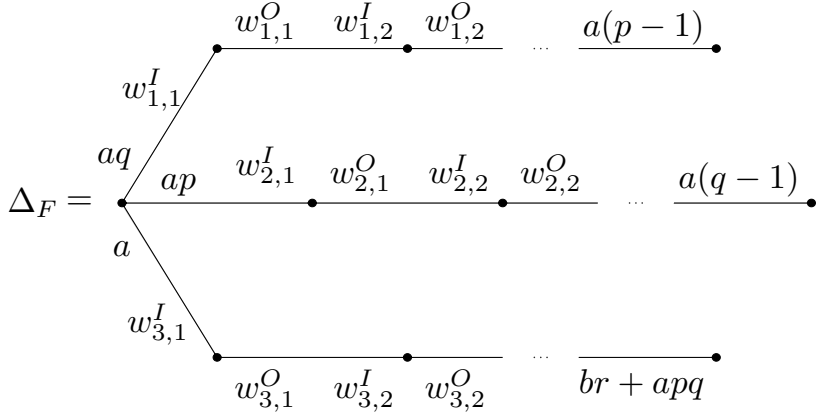
\begin{figure}[H]
\begin{center}
\begin{tikzpicture}[xscale=.8, yscale=.75]
\filldraw[black, thin]
           (-42pt,0pt) node[scale=1.3] {$\Delta_F$}
           (-20pt,0pt) node[scale=1.3] {$=$}
           (0,0) circle (1.9pt) node[above right=8pt and -14pt,scale=1.3] {$aq$}
                                node[above right=-2pt and 10pt,scale=1.3] {$ap$}
                                node[below=10pt,scale=1.3] {$a$}
      -- ++(60:90pt) circle (1.9pt) node[below left=3pt and 8pt,scale=1.3] {$w_{1,1}^I$}
                                    node[above right=-3pt and 2pt,scale=1.3] {$w_{1,1}^O$}
      -- ++(90pt,0pt) circle (1.9pt) node[above left=-3pt and 0pt,scale=1.3] {$w_{1,2}^I$}
                                     node[above right=-3pt and 2pt,scale=1.3] {$w_{1,2}^O$}
      -- ++(45pt,0pt)
         ++(15pt,0pt) circle (.1pt)
         ++(3pt,0pt) circle (.1pt)
         ++(3pt,0pt) circle (.1pt)
         ++(15pt,0pt)
      -- ++(65pt,0pt) circle (1.9pt) node [above left=-3pt and 0pt,scale=1.3] {$a(p-1)$}
(0,0) -- ++(0:90pt) circle (1.9pt) node[above left=1pt and 1pt,scale=1.3] {$w_{2,1}^I$}
                                     node[above right=-2pt and 2pt,scale=1.3] {$w_{2,1}^O$}
      -- ++(90pt,0pt) circle (1.9pt) node [above left=-2pt and 0pt,scale=1.3] {$w_{2,2}^I$}
                                     node [above right=-2pt and 2pt,scale=1.3] {$w_{2,2}^O$}
      -- ++(45pt,0pt)
         ++(15pt,0pt) circle (.1pt)
         ++(3pt,0pt) circle (.1pt)
         ++(3pt,0pt) circle (.1pt)
         ++(15pt,0pt)
      -- ++(65pt,0pt) circle (1.9pt) node [above left=-2pt and 0pt,scale=1.3] {$a(q-1)$}
  (0,0) -- (-60:90pt) circle (1.9pt) node[above left=0pt and 6pt,scale=1.3] {$w_{3,1}^I$}
                                     node[below right=0pt and 2pt,scale=1.3] {$w_{3,1}^O$}  
      -- ++(90pt,0pt) circle (1.9pt) node[below left,scale=1.3] {$w_{3,2}^I$}
                                     node[below right=0pt and 2pt,scale=1.3] {$w_{3,2}^O$}
      -- ++(45pt,0pt)
         ++(15pt,0pt) circle (.1pt)
         ++(3pt,0pt) circle (.1pt)
         ++(3pt,0pt) circle (.1pt)
         ++(15pt,0pt)
      -- ++(65pt,0pt) circle (1.9pt) node [below left,scale=1.3] {$br+apq$}
;
\end{tikzpicture}
\end{center}
\caption{Maximal splice diagram $\Delta_F$ with weights computed.}
\label{spldpart}
\end{figure}

Let $d$ be the entry of matrix $D$ corresponding to the node of the graph. Let us compute $d$ (using Remark~\ref{detpath}):
\begin{align*}
d &= - \left( - \frac{a^3pq}{\ad} + \frac{a^2 p}{\ad} + \frac{a^2 q}{\ad} + \frac{a^2 pq}{\ad} \right) \\
  &= - \left( \frac{-a^2 \left( (a-1) pq -p-q \right)}{\ad} \right)= \frac{(a-1) pq -p-q}{\delta} \\
  &= \frac{(a-2)pq + b \delta}{\delta} = \frac{(a-2)pq}{\delta} + b
\end{align*}

Now we have three subcases to study:
\begin{enumerate}[i')]
\item If $a \neq 2$ and $a \not\equiv 2 \pmod{\delta}$,
\item if $a \neq 2$ and $a \equiv 2 \pmod{\delta}$,
\item if $a=2$.
\end{enumerate}

\subsubsection*{Subcase (i')}

If $a \neq 2$ and $a \not\equiv 2 \pmod{\delta}$, we claim that $d$ is not an integer.

Indeed, if $\delta \mid p$, as $\delta \mid (pq-p-q)$, then $\delta \mid q$, but $(p,q)=1$. Then $\delta \nmid p$. By the same argument, $\delta \nmid q$.

If $\delta$ is prime, then follows that $\delta \nmid pq$ and since $a-2 \not\equiv 0 \pmod{\delta}$, then $\delta \nmid a-2$ and $d$ is not an integer.

Now, if $\delta$ is not prime, let $\delta_1$ be a factor of $\delta$ such that $\delta_1 \mid a-2$ (if there is not such factor, it follows the previous argument). Let $\delta = \delta_1 \delta_2$. Then $d$ is integer if and only if $\delta_2 \mid pq$.

Analogously to $\delta$, $\delta_2 \nmid p$ and $\delta_2 \nmid q$. Moreover, if there exist $d_3$ such that $1 < d_3 < d_2$ and $d_3 \mid p$, then the same argument tells us that $d_3 \mid q$, which is a contradiction.

Then we obtain that if $a \not\equiv 2 \pmod{\delta}$, then $d$ is not integer and therefore the class $D$ has not integer coefficients.

\subsubsection*{Subcase (ii')}

If $a \neq 2$ and $a \equiv 2 \pmod{\delta}$,  then $d \in \Z{}$.

Let $a=c \delta+2$, and let $d_{1}$ and $d_{2}$ be the entries of matrix $D$ corresponding to the leaves with weights $a(p-1)$ and $a(q-1)$ respectively (as before). By equation~\eqref{coeffd} and Remark~\ref{detpath} we have that
\begin{align*}
d_{1} &= - \left(- \frac{a^2 p}{\ad} + \frac{ap-a}{\ad} + \frac{a}{\ad} + \frac{ap}{\ad} \right) \\
        &= - \frac{-a^2 p +2ap}{\ad} = \frac{(a-2) p}{a \delta} = \frac{cp}{a} \\
d_{2} &= \frac{cq}{a} \ . \\
\end{align*}
Then we claim that $d_{1}$ and $d_{2}$ cannot both be integers.

In order to have $d_{1}$ and $d_{2}$ integers, we should have that $a \mid cp$ and $a \mid cq$, then 
let $a_1 \in \Z{}$ be such that $a=a_1 a_2$ with $a_2 \mid c$ and $a_1 \mid p$, $a_1 \mid q$ which it is not possible since $(p,q)=1$.

Then we obtain that if $a \neq 2$ and $a \equiv 2 \pmod{\delta}$, the class $D$ has not integer coefficients.

\subsubsection*{Subcase (iii')}

Let $a=2$. We compute the Seifert invariants and the splice diagram $\Delta_F$ in this particular case. The Seifert invariants are:
\begin{equation*}
\Bigl(\frac{\delta-1}{2}; \ - \frac{\delta}{2pq}; \ (2q, \beta_1),\,(2p, \beta_2),\,(2,1)\Bigr) \ ,
\end{equation*}
where
\begin{align*}
b \beta_1 & \equiv -1 \pmod{2q} \ ,\\
b \beta_2 & \equiv -1 \pmod{2p} \ . \\
\end{align*}
and the associate splice diagram $\Delta_F$ is given by Figure~\ref{spldados}.
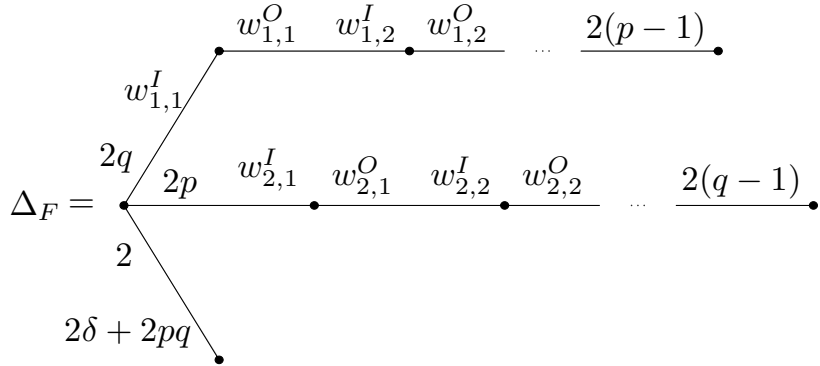
\begin{figure}[H]
\begin{center}
\begin{tikzpicture}[xscale=.8, yscale=.75]
\filldraw[black, thin]
           (-42pt,0pt) node[scale=1.3] {$\Delta_F$}
           (-20pt,0pt) node[scale=1.3] {$=$}
           (0,0) circle (1.9pt) node[above right=8pt and -14pt,scale=1.3] {$2q$}
                                node[above right=-2pt and 10pt,scale=1.3] {$2p$}
                                node[below=10pt,scale=1.3] {$2$}
      -- ++(60:90pt) circle (1.9pt) node[below left=3pt and 8pt,scale=1.3] {$w_{1,1}^I$}
                                    node[above right=-3pt and 2pt,scale=1.3] {$w_{1,1}^O$}
      -- ++(90pt,0pt) circle (1.9pt) node[above left=-3pt and 0pt,scale=1.3] {$w_{1,2}^I$}
                                     node[above right=-3pt and 2pt,scale=1.3] {$w_{1,2}^O$}
      -- ++(45pt,0pt)
         ++(15pt,0pt) circle (.1pt)
         ++(3pt,0pt) circle (.1pt)
         ++(3pt,0pt) circle (.1pt)
         ++(15pt,0pt)
      -- ++(65pt,0pt) circle (1.9pt) node [above left=-3pt and 0pt,scale=1.3] {$2(p-1)$}
(0,0) -- ++(0:90pt) circle (1.9pt) node[above left=1pt and 1pt,scale=1.3] {$w_{2,1}^I$}
                                     node[above right=-2pt and 2pt,scale=1.3] {$w_{2,1}^O$}
      -- ++(90pt,0pt) circle (1.9pt) node [above left=-2pt and 0pt,scale=1.3] {$w_{2,2}^I$}
                                     node [above right=-2pt and 2pt,scale=1.3] {$w_{2,2}^O$}
      -- ++(45pt,0pt)
         ++(15pt,0pt) circle (.1pt)
         ++(3pt,0pt) circle (.1pt)
         ++(3pt,0pt) circle (.1pt)
         ++(15pt,0pt)
      -- ++(65pt,0pt) circle (1.9pt) node [above left=-2pt and 0pt,scale=1.3] {$2(q-1)$}
  (0,0) -- (-60:90pt) circle (1.9pt) node[above left=0pt and 6pt,scale=1.3] {$2 \delta + 2pq$}
;
\end{tikzpicture}
\end{center}
\caption{Maximal splice diagram $\Delta_F$ when $a=2$.}
\label{spldados}
\end{figure}

Then, by equation~\ref{coeffd} and Remark~\ref{detpath},
\begin{align*}
d &=- \left( - \frac{8pq}{4\delta} + \frac{4p}{4\delta} + \frac{4q}{4\delta} + \frac{4pq}{4\delta} \right) = - \frac{-4pq+4p+4q}{4\delta} \\
  &= - \frac{-4(pq-p-q)}{4\delta} = \frac{b \delta}{\delta} = b \\
d_{1} &= - \left( - \frac{4p}{4\delta} + \frac{2(p-1)}{4\delta} + \frac{2}{4\delta} + \frac{2p}{4\delta} \right) = - \frac{-4p + 2p -2 +2 +2p}{4\delta}  \\
  &= 0 \\
d_{2} &= - \left( - \frac{4q}{4\delta} + \frac{2(q-1)}{4\delta} + \frac{2}{4\delta} + \frac{2q}{4\delta} \right) = - \frac{-4q + 2q -2 +2 +2q}{4\delta}  \\
  &= 0 \\
d_{3} &= - \left( - \frac{4pq}{4\delta} + \frac{2\delta+2pq}{4\delta} + \frac{2p}{4\delta} + \frac{2q}{4\delta} \right) = - \frac{2\delta -4pq+2pq+2p+2q}{4\delta} \\
  &= - \frac{2\delta -2b\delta}{4\delta} = - \frac{1-b}{2} = \frac{b-1}{2} \ ,
\end{align*}
where $d_i$ is the entry of the matrix associated to the $i$-th leaf of the graph.

Thus, considering equation~\eqref{kend}, we obtain that the coefficients of the canonical class are given by
\begin{align*}
k         &= - (b)-1= -b-1 \\
k_{1,s_1} &= - (0)-1=-1 \\
k_{2,s_2} &= - (0)-1=-1 \\
k_3       &= -(\frac{b-1}{2}) -1 = - \frac{b+1}{2} \ . \\
\end{align*}

By the linear system given by the adjunction formula, we have that if $k_{i,s_i}$ is integer, then $k_{i,j}$ is integer for $i=1,2$ and $1 \leq j < s_i$. Thus, in this case, all the coefficients of $K$ are integers.
\end{proof}

%
%

\begin{cor}
If $a \neq 2$, does not exist a complex analytic germ $G \colon (\C^3,0) \to (\C,0)$ with isolated singularity at the origin such that the link $L_G$ is isomorphic to the link $L_F$; \ie the open-book decomposition of the sphere $\Sp{5}$ given by the Milnor fibration of $F$ is not given by the Milnor fibration of a normal Gorenstein complex hypersurface singularity $(V,p)$.
\end{cor}

When $a=2$, we would like to know if, given a complex analytic germ $G \colon (\C^3,0) \to (\C,0)$ such that the link $L_G$ is isomorphic to the link $L_F$, the Milnor fibrations are the same or not. If they are not, we would have new open book decompositions of $\Sp{5}$. 

In general we have not been able to find an answer, however, in the case $a=r=2$, we know a complex singularity that realises the link $L_F$, then it is possible to compare directly the corresponding Milnor fibrations and we obtain that they are not equivalent.

\begin{thm}\label{opb2}
Let $F \colon (\C^3,0) \cong (\R^6,0) \to (\C,0) \cong (\R^2,0)$ be the real polynomial function defined by
\begin{equation*}
F(x,y,z)= \overline{xy}(x^p+y^q) + z^2 \ .
\end{equation*}
Let $G \colon (\C^3,0) \cong (\R^6,0) \to (\C,0) \cong (\R^2,0)$ be the complex polynomial function defined by
\begin{equation*}
G(x,y,z)= xy(x^p+y^q) + z^2 \ .
\end{equation*}
Then the link $L_F$ is homeomorphic to the link $L_G$ but the corresponding Milnor fibrations are not equivalent; \ie the corresponding open book decompositions of $\Sp{5}$ are not equivalent.
\end{thm}
\begin{proof}
In this case, in order to prove that $L_F$ is homeomorphic to $L_G$, we compute the Seifert invariants in both cases.

Since $r=2$, $a=2$ and $\delta$ is odd, then $\delta$ can only be $1$. Then, by Theorem~\ref{thprinc}, the Seifert invariants of $L_F$ are
\begin{equation*}
\Bigl(0; \ - \frac{1}{2pq}; \ (2q, \beta_1),\,(2p, \beta_2),\,(2, 1)\Bigr) \ ,
\end{equation*}
where
\begin{align*}
b \beta_1 & \equiv -1 \pmod{2q} \ ,\\
b \beta_2 & \equiv -1 \pmod{2p} \ .\\
\end{align*}
Also , since $a=2$, we have that $\beta_3=1$.

Now, let $G \colon (\C^3,0) \to (\C,0)$ be given by 
\begin{equation*}
G(x,y,z)=xy(x^p+y^q)+z^2
\end{equation*}
with $p,q \in \Z{}^{+}$, $p,q \geq 2$ and $\gcd(p,q)=1$. Let $L_G$ be the link of the singularity at the origin of $\pr{G}(0)$.

We can compute the Seifert invariants of $L_G$ analogously to the computation of the Seifert invariants of $L_F$. In this case, let $\delta'$ be the $\gcd(pq+p+q,2)=1$ and let $b'=(pq+p+q)/\delta$. Then, the genus of the surface of orbits is $(\delta'-1)/2=0$ and the rational Euler class is given by $-\delta'/2pq=-1/2pq$.

Then the Seifert invariants of $L_G$ are
\begin{equation*}
\Bigl(0; \ - \frac{1}{2pq}; \ (2q, \beta'_1),\,(2p, \beta'_2),\,(2, 1)\Bigr) \ ,
\end{equation*}
where
\begin{align*}
b' \beta'_1 & \equiv 1 \pmod{2q} \ ,\\
b' \beta'_2 & \equiv 1 \pmod{2p} \ .\\
\end{align*}
But, let us take $\beta_1$ of the first Seifert pair of $L_F$, we have that:
\begin{align*}
b \beta_1                     & \equiv -1 \pmod{2q} \\
(pq-p-q) \beta_1              & \equiv -1 \pmod{2q} \\
\left(q(p-1)-p\right) \beta_1 & \equiv -1 \pmod{2q} \\
q(p-1) \beta_1 - p \beta_1    & \equiv -1 \pmod{2q} \\
2q(p-1) \beta_1 - 2p \beta_1  & \equiv -2 \pmod{2q} \\
                - 2p \beta_1  & \equiv -2 \pmod{2q} \\
\end{align*}
Now
\begin{align*}
b' \beta'_1                         & \equiv 1 \pmod{2q} \\
(pq+p+q) \beta'_1                   & \equiv 1 \pmod{2q} \\
(b + 2p +2q) \beta'_1               & \equiv 1 \pmod{2q} \\
b \beta'_1 +2p \beta'_1 +2q\beta'_1 & \equiv 1 \pmod{2q} \\
b \beta'_1 +2p \beta'_1 & \equiv 1 \pmod{2q} \\
\end{align*}
But then, by the previous equality for $\beta_1$, we have that
\begin{align*}
- 2p \beta_1  & \equiv -2 \pmod{2q} \\
b \beta_1 +2p \beta_1 & \equiv -1+2 \pmod{2q} \\
b \beta_1 +2p \beta_1 & \equiv 1 \pmod{2q} \\
\end{align*}
As $0 \leq \beta_1 < \alpha_1$ and $0 \leq \beta'_1 < \alpha_1$, we have that $\beta_1=\beta'_1$. Analogously for $\beta_2$ and $\beta'_2$.

Thus the Seifert invariants for $L_G$ are:
\begin{equation*}
\Bigl(0; \ - \frac{1}{2pq}; \ (2q, \beta_1),\,(2p, \beta_2),\,(2, 1)\Bigr) \ ,
\end{equation*}
where
\begin{align*}
b \beta_1 & \equiv -1 \pmod{2q} \ ,\\
b \beta_2 & \equiv -1 \pmod{2p} \ .\\
\end{align*}
Then the link $L_F$ is homeomorphic to the link $L_G$.

However, in this case, in \cite{PASJ03} it is established that the Milnor fbrations of the functions $f$ and $g$ given by
\begin{align*}
f(x,y) & = \overline{xy}(x^p+y^q) \\ 
g(x,y) & = xy(x^p+y^q) \\  
\end{align*}
are different because the corresponding Milnor fibres have different homotopy type. It is known that both fibres have the homotopy type of wedges of circles; let $\mu_f$ and $\mu_g$ be the corresponding number of spheres in the wedges, then $\mu_f \neq \mu_g$.

We apply the Join theorem for polar weighted homogeneous polynomials (see \cite[Theorem~4.1]{Cis09}), which says that the Milnor fibre of the sum (over independent variables) of two polar weighted homogeneous polynomials is homotopically equivalent to the join of the Milnor fibres of the polar weighted homogeneous polynomials. Then the Milnor fibre $\F_F$ has the homotopy type of a wedge of $\mu_f$ $2$-spheres and the Milnor fibre $\F_G$ has the homotopy type of a wedge of $\mu_g$ $2$-spheres; thus the Milnor fibrations of $F$ and $G$ are different.
\end{proof}

However, it remains open the following questions: There exists a complex germ $H$ such that $H$ is different from $G$ but the link $L_H$ is homeomorphic to $L_F$? If it is the case, how are the corresponding Milnor fibrations?



\section{Examples}
In \cite{MR2922705}, we obtain from the Laufer-Steenbrink formula a congruence that can be used as an obstruction for the equivalence of Milnor fibrations of $F$ and the germ of a complex singularity from $\C^3$ to $\C$. In this section we illustrate how we use such congruence in the case $a=2$ and $b=1$.

Let $b=1$; i.e. $pq-p-q=\delta$ and let $r=a \delta$ with $a=2$.

In this case, the Seifert invariants of $L_F$ are:
\begin{equation*}
\left(\frac{\delta-1}{2}; -\frac{\delta}{2pq}; (2q, 2q-1), (2p, 2p-1), (2,1) \right) \ .
\end{equation*}

Let us compute now the plumbing graph $\Gamma_F$ in this case:
\begin{align*}
\frac{\alpha_1}{\alpha_1 - \beta_1} &= \frac{2q}{2q - (2q-1)} = 2q = [2q]\\
\frac{\alpha_2}{\alpha_2 - \beta_2} &= \frac{2p}{2p - (2p-1)} = 2p = [2p]\\
\frac{\alpha_3}{\alpha_3 - \beta_3} &= \frac{2}{2-1}= 2 = [2] \ . \\
\end{align*}

The Euler number at the node is given by
\begin{align*}
e &= -\frac{\delta}{2pq} + \frac{2q-1}{2q} + \frac{2p-1}{2q} + \frac{1}{2} -3 =\frac{-\delta + 2pq-p + 2pq - q + pq}{2pq} -3\\
  &=\frac{- \delta+(pq-p-q)+4pq}{2pq} -3 =\frac{4pq}{2pq}-3 = 2-3 = -1
\end{align*}

Thus, the graph $\Gamma_F$ is as in Figure~\ref{plgrex}.
\begin{figure}[H]
\begin{center}
\begin{tikzpicture}[xscale=.8,yscale=.7]
\filldraw[black, thin] 
(-50pt,0pt) node {$\Gamma_F$}
           (-30pt,0pt) node {$=$}
(0,0) circle (1.9pt) node[distance=.5pt and 2pt,below left,scale=1.3] {$-1$}
              -- ++(40:80pt) circle (1.9pt) node [above,scale=1.3] {$-2q$}
        (0,0) -- (61.28pt,0pt) circle (1.9pt) node [above,scale=1.3] {$-2p$}
        (0,0) -- (-40:80pt) circle (1.9pt) node [below=.5pt,scale=1.3] {$-2$}
;
\end{tikzpicture}
\end{center}
\caption{Plumbing graph $\Gamma_F$ in the case $a=2$ and $b=1$.}
\label{plgrex}
\end{figure}
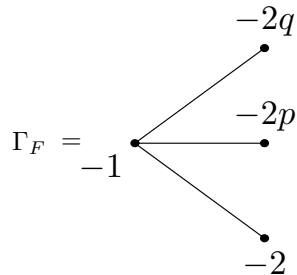

Now we compute the splice weights associated at each edge and we obtain the corresponding maximal splice diagram $\Delta_F$ (see Figure~\ref{spldiagex}).

\begin{figure}[H]
\begin{center}
\begin{tikzpicture}[xscale=.8, yscale=.75]
\filldraw[black, thin]
           (-40pt,0pt) node {$\Delta_F$}
           (-20pt,0pt) node {$=$}
           (0,0) circle (1.9pt) node[above right=8pt and -14pt,scale=1.3] {$2q$}
                                node[above right=0pt and 13pt,scale=1.3] {$2p$}
                                node[below=10pt,scale=1.3] {$2$}
      -- ++(60:90pt) circle (1.9pt) node[below left=3pt and 6pt,scale=1.3] {$2(p-1)$}
(0,0) -- ++(0:120pt) circle (1.9pt) node[above left=0pt and 2pt,scale=1.3] {$2(q-1)$}
(0,0) -- (-60:90pt)   circle (1.9pt) node[above left=0pt and 6pt,scale=1.3] {$2(\delta+pq)$}
;
\end{tikzpicture}
\end{center}
\caption{Maximal splice diagram $\Delta_F$ in the case $a=2$ and $b=1$.}
\label{spldiagex}
\end{figure}
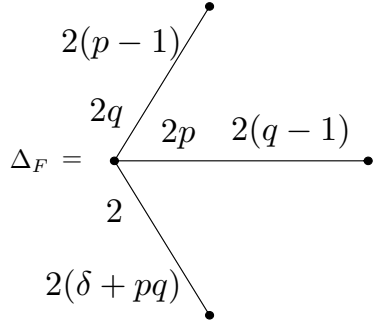

Then by the proof of Theorem~\ref{coeffint} the coefficients of $D$ are
\begin{align*}
d=  1 \quad \quad d_1= 0 \quad \quad d_2= 0 \quad \quad d_3= 0
\end{align*}
and by equation~\eqref{kend} the coefficients of $K$ are
\begin{align*}
k= -(1) -1= -2 &\quad \quad k_1= - (0)-1=-1 \\
k_2= - (0)-1=-1 &\quad \quad k_3 = - (0)-1=-1
\end{align*}

Since the canonical class $K$ has integer coefficients, there could exists a complex analytic germ $G \colon (\C^3,0) \to (\C,0)$ with isolated singularity with link $L_G$ homeomorphic to the link $L_F$.

As it is done in \cite[\S~5]{MR2922705}, in order to see if the  Milnor fibre $\mathcal{F}$ is diffeomorphic to a smoothing  (see \cite[Definition~9]{MR2922705}) of a complex surface singularity we compute the Euler characteristic of $\mathcal{F}$ and the Euler characteristic of the resolution $\tilde{V}_F$  and we apply Theorem \ref{cong}, which can be seen as a consequence of the Laufer-Steenbrink formula (see \cite{MR0450287} and \cite{MR713277}).

\begin{thm}(see for instance \cite[\S~4, Corollary~1]{MR713273})\label{cong}
Let $(V,0)$ be a normal Gorenstein complex surface singularity with link $L_V$. If $(V,0)$ is smoothable, then one has
\begin{equation*}
\chi (L_V) + K^2_V \equiv \chi(V^{\prime}) \pmod{12}
\end{equation*}
where $V^{\prime}$ is a smoothing of $V$ and $K_L$ is the canonical class of a resolution of $V$.
\end{thm}

We also have that
\begin{equation*}
\chi(\tilde{V})= \sum_{i=1}^n \chi(E_i) - \sum_{i<j} E_i \cdot E_j \ , 
\end{equation*}
where $E_i$ is one of the irreducible components of $\tilde{V}$ and
\begin{equation*}
K^2 = K^T A_F K \ .
\end{equation*}

Then, as the plumbing graph $\Gamma_F$ has 4 vertices and 3 edges, we have that
\begin{equation*}
\chi(\tilde{V})= 2(4) - 3 = 5
\end{equation*}
and
\begin{equation*}
K^2 =
\begin{pmatrix}
-2 & -1 & -1 & -1 \\
\end{pmatrix}
\begin{pmatrix}
-1 &   1 &   1 & 1  \\
1  & -2q &   0 & 0  \\
1  &   0 & -2p & 0  \\
1  &   0 &   0 & -2 \\
\end{pmatrix}
\begin{pmatrix}
       -2 \\
       -1 \\
       -1 \\
       -1 \\
      \end{pmatrix}
= -2q -2p +6 \ .
\end{equation*}

Then
\begin{equation*}
\chi(\tilde{V}) + K^2 = 5 -2q -2p +6 = 11 -2p -2q
\end{equation*}

Now, let $f \colon \C^2 \to \C$ given by $f(x,y)=\overline{xy}(x^p+y^q)$. Analogously to \cite[Example~1]{MR2115674}, the resolution graph has one node with multiplicity $pq-p-q=\delta$, then  
$\mathcal{F}_f = \bigvee_{i=1}^t \Sp{1}_i$ where $\mathcal{F}_f$ is the Milnor fibre of $f$ and $t= 1- \chi(\mathcal{F}_f)$.

Now we compute the Euler characteristic of $\mathcal{F}$ using the Join Theorem for polar weighted homogeneous polynomials and we obtain that
\begin{align*}
\chi(\mathcal{F}) &= 1 + (r-1) \left(1- \chi(\mathcal{F}_f)\right)= 1+ (2 \delta-1)(1-(-pq+p+q)) \\
                  &= 1+(2 \delta-1)(1+ \delta) = 1 + 2 \delta + 2 \delta^2 -1 - \delta \\
                  &= \delta (2 \delta + 1) \ .
\end{align*}

Using Theorem~\ref{cong} we would like to see when
\begin{equation*}
\chi(\tilde{V}) + K^2 - \chi(\mathcal{F}) = 11 -2p -2q - \left(\delta(2 \delta + 1)\right) \equiv 0 \pmod{12} \ .
\end{equation*}

Let $p \equiv 0 \pmod{12}$, then
\begin{equation*}
\chi(\tilde{V}) + K^2 - \chi(\mathcal{F}) \equiv -2q^2 -q + 11 \pmod{12}.
\end{equation*}
If we look when $q \equiv 0, \ldots 11 \pmod{12}$, we obtain that
\begin{equation*}
\chi(\tilde{V}) + K^2 - \chi(\mathcal{F}) \not\equiv 0 \pmod{12}
\end{equation*}
in any case.

Analogously, we consider $p \equiv 1, \ldots, 11 \pmod{12}$ and in each case we look when $q \equiv 0, \ldots 11 \pmod{12}$ and we obtain

\begin{prop}
Given $F$ as above, let $a=2$ and $b=1$. Then
\begin{equation*}
\chi(\tilde{V}) + K^2 - \chi(\mathcal{F}) \equiv 0 \pmod{12}
\end{equation*}
when
\begin{enumerate}[i)]
\item $p \equiv 1 \pmod{12}$ and $q \equiv 4,10 \pmod{12}$,
\item $p \equiv 2 \pmod{12}$ and $q \equiv 5 \pmod{12}$,
\item $p \equiv 4 \pmod{12}$ and $q \equiv 1 \pmod{12}$,
\item $p \equiv 5 \pmod{12}$ and $q \equiv 2,8 \pmod{12}$,
\item $p \equiv 8 \pmod{12}$ and $q \equiv 5 \pmod{12}$,
\item $p \equiv 10 \pmod{12}$ and $q \equiv 1 \pmod{12}$.
\end{enumerate}
\end{prop}

Thus, in these cases we cannot use the previous congruence as an obstruction to have a complex singularity from $\C^3$ to $\C$ with Milnor fibration equivalent to the Milnor fibration of $F$. However, for all the remaining cases we have that the Milnor fibration of $F$ cannot appear as Milnor fibration of a complex singularity from $\C^3$ to $\C$.

It is our intention to study other invariants of singularities that could give us a complete answer to the question for which functions $F$, the Milnor fibration of $F$ appear as the Milnor fibration of a complex singularity.

\end{document}